\documentclass[pvipdfmx]{amsart}
\usepackage{amsmath,amsfonts,amssymb,amscd,amsthm}
\newcommand{\N}{\mathbb{N}}
\newcommand{\Z}{\mathbb{Z}}

\newcommand{\R}{\mathbb{R}}
\newcommand{\C}{\mathbb{C}}

\newcommand{\bra}[1]{\left( {#1} \right)}

\newcommand{\sqb}[1]{\left[{#1} \right]}

\newcommand{\nor}[1]{\left\| {#1} \right\|}

\newcommand{\ro}[1]{\sqrt{\mathstrut{#1}}}

\newcommand{\m}[1]{\mbox{\boldmath $#1$}}

\newtheorem{thm}{Theorem}[section]
\newtheorem{lem}[thm]{Lemma}

\theoremstyle{remark}
\newtheorem{rem}[thm]{Remark}
\newtheorem{defi}[thm]{Definition}
\newtheorem{ass}[thm]{Assumption}

\newcommand{\diag}{\operatorname{diag}}
\newcommand{\I}{\infty}
\newcommand{\na}{\nabla}
\newcommand{\fy}{\varphi}
\newcommand{\F}{\mathcal{F}}
\newcommand{\pq}{\quad}
\newcommand{\p}{\partial}
\newcommand{\e}{\epsilon}
\newcommand{\De}{\Delta}
\newcommand{\om}{\omega}
\newcommand{\z}{\zeta}
\newcommand{\al}{\alpha}
\newcommand{\supp}{\operatorname{supp}}
\newcommand{\ti}{\tilde}
\newcommand{\lec}{\lesssim}
\newcommand{\EQ}[1]{\begin{equation}\begin{split} #1 \end{split}\end{equation}}

\newcommand{\LR}[1]{\langle #1 \rangle}

\numberwithin{equation}{section}

\title[Randomized final-data problem for S-NLS and GP]
{Randomized final-data problem for\\
	Systems of Nonlinear Schr\"odinger Equations\\
	and the Gross-Pitaevskii Equation}
\author{Kenji Nakanishi
\and Takuto Yamamoto}
\address{Research Institute for Mathematical Sciences 
Kyoto University, Kyoto 606-8502, Japan}
\email{kenji@kurims.kyoto-u.ac.jp}
\address{Research Institute for Mathematical Sciences 
Kyoto University, Kyoto 606-8502, Japan}
\email{takutoya@kurims.kyoto-u.ac.jp}
\date{\today}

\begin{document}

\maketitle

\begin{abstract}
We consider the final-data problem for
systems of nonlinear Schr\"odinger equations \eqref{eq:NLS} with $L^2$ subcritical nonlinearity. 
An asymptotically free solution is uniquely obtained for almost every randomized asymptotic profile in $L^2(\R^d)$, extending the result of J.~Murphy \cite{Murphy} to powers equal to or lower than the Strauss exponent. 
In particular, systems with quadratic nonlinearity can be treated in three space dimensions, and by the same argument, the Gross-Pitaevskii equation in the energy space. 
The extension is by use of the Strichartz estimate with a time weight. 
\end{abstract}

\section{Introduction}
\label{sec:intro}

We study asymptotic behavior for large time of solutions to a system of nonlinear Schr\"odinger equations: \begin{align}
&\tag{NLS}\label{eq:NLS}
i\partial_{t}\m{u}+M\Delta{\m{u}}=\m{f}(\m{u}),
\end{align}
where $\m{u}=(u_{1},\dots,u_{N})(t,x):\R\times\R^{d}\to\C^{N}$ for some $d,N\in\N$ is the unknown function, 
$M=\diag(M_1,\dots,M_N)$ is an $N\times N$ diagonal matrix with non-zero real diagonal elements,
and $\m{f}=(f_{1},\dots,f_{N}):\C^{N}\to\C^{N}$ is a power-type nonlinearity. 
The precise conditions are given in Assumption \ref{ass:1} below.

We are interested in whether the solution $u$ of \eqref{eq:NLS} can be approximated by a free solution for large time, namely the nonlinear scattering problem. 
Let 
\EQ{ \label{def Ut}
 U(t):=e^{itM\Delta}} 
denote the free propagator, and let $X$ be a Banach space embedded in the space of $\C^N$-valued tempered distributions on $\R^d$. We say that a solution $u$ of \eqref{eq:NLS} scatters in $X$ as $t\to\I$ if $U(-t)u(t)$ converges to some final-state $u_+$ in $X$ as $t\to\I$. 

The scattering problem can be divided into two parts: (1) For a given final-state $u_+$, if there is a scattering solution, and (2) For a given initial data $u(0)$, if the solution $u$ scatters. 
We will call (1) the final-data problem, and (2) the initial-data problem. 
There is a huge amount of literature on both. For a brief review, we take the simplest case of \eqref{eq:NLS}, which is the scalar equation with a pure power nonlinearity: for some constants $p>0$ and $\lambda\in\R$, 
\begin{align}
&\label{eq:power-NLS}
i\partial_{t}u+\Delta{u}=\lambda|u|^{p}u. 
\end{align}
This equation is invariant for the scaling with a parameter $a>0$: 
\begin{align}
&
u(t,x)\mapsto a^{\frac{2}{p}}u(a^{2}t,a x). 
\label{def:scaling}
\end{align}

In the mass critical case $p=4/d$, the above scaling preserves the $L^2(\R^d)$ norm of $u$ (the mass), which is a conserved quantity for \eqref{eq:power-NLS}. 
In general, the scattering problem is more difficult for lower $p$, since the free solutions decay by dispersion for large time, although the local existence is more difficult for higher $p$. 
In this paper, we restrict our attention to the $L^2$ subcritical case $p<4/d$, where \eqref{eq:power-NLS} is known to be globally well-posed in $L^2(\R^d)$ for any $\lambda\in\R$, see \cite{Tsutsumi-1987}.  
On the other hand, the scattering is known to fail for $p\le 2/d$, see \cite{Glassey-1973,Strauss,Barab}, so we should restrict to $p>2/d$. 

On the initial-data problem, Tsutsumi and Yajima \cite{T-Y} showed that
if $p>2/d$, $(d-2)p<4$ and $\lambda\ge 0$, then for any initial-data $u(0)$ in the weighted Sobolev space 
\EQ{
 \Sigma:=\{\fy\in L^2(\R^d) \mid \na\fy, x\fy\in L^2(\R^d)\},} 
the solution $u$ scatters in $L^2(\R^d)$. 
If $p\ge p_0(d):=\frac{2-d+\sqrt{d^{2}+12d+4}}{2d}$, then the solution $u$ scatters also in $\Sigma$, see \cite{Tsutsumi-1985,C-W}. This critical number $p_0(d)\in(2/d,4/d)$ is often called by the name of Strauss, for his systematic study \cite{Strauss-1981} of the scattering problem for small data in $L^{(p+2)'}$. 
For small initial data in the weighted space 
\EQ{
 \F H^s(\R^d) := \{ \fy \in L^2_{loc}(\R^d) \mid (1+|x|)^s\fy\in L^2(\R^d)\},}
the solution $u$ scatters in $\F H^s$ if $0\le s<\min(d/2,p)$ and $4/(d+2s)\le p\le 4/d$, see \cite{GOV,NO}. 

On the final-data problem, the first author proved \cite{Nakanishi} existence of a solution $u$ scattering in $L^2(\R^d)$ for any final-data $u_+\in L^2(\R^d)$ for $p\in(2/d,4/d)$ and $d\ge 3$, which was extended to $d=2$ by Holmer and Tzirakis \cite{HT} in the case $u_+\in H^1(\R^2)$ and $\lambda\ge 0$. 
However, the uniqueness of the solution $u$ for a given final-state $u_+$ is an open question in those results, since the proof relies crucially on a compactness argument. 
For $u_+\in\F H^s$, the solution $u$ scattering to $u_+$ in $\F H^s$ uniquely exists, provided that $0<s<\min(d/2,p)$ and $4/(n+2s)\le p<4/d$, see \cite{C-W,Nakanishi}. 

It is worth noting that the uniqueness problem in \cite{Nakanishi} is super-critical in view of the scaling, where the invariant scaling \eqref{def:scaling} implies that neither smallness in $L^2(\R^d)$ nor the limit $t\to\I$ makes the situation better. Indeed, the problem in the region 
\EQ{
 \|u_+\|_{L^2(\R^d)} \le 1, \pq \|u\|_{L^2(\R^d)}\le 1, \pq t\ge 1}
can be reduced by the invariant scaling to the smaller region
\EQ{
  \|u_+\|_{L^2(\R^d)} \le \e, \pq \|u\|_{L^2(\R^d)}\le \e, \pq t\ge 1/\e}
for any $\e>0$, which implies that the latter region is no easier than the former.   
Another way to observe the supercriticality is to invoke the pseudo-conformal inversion $u(t,x)\mapsto(it)^{-d/2}\exp(i|x|^2/(4t))\bar{u}(1/t,x/t)$, which transforms the final-data problem into the Cauchy problem 
\EQ{
 i\dot u + \Delta u = \lambda t^{2-dp/2}|u|^p u, \pq u(0)=\overline{\F u_+} \in L^2(\R^d),}
where $\F$ denotes the Fourier transform unitary on $L^2(\R^d)$. 
In the same way as above, the invariant scaling implies that the local Cauchy problem restricted to  
$\|u\|_{L^2(\R^d)}\le \e$ and $0<t<\e$ 
is no easier for any $\e>0$ than that with $\e=1$. 

Recently, Murphy \cite{Murphy} shed new light on this supercritical problem. 
For arbitrary given final-data $u_+\in L^2(\R^d)$, he introduced a randomization in the physical space, and proved that a scattering solution with some space-time integrability exists uniquely for almost every randomized final-data. 
In view of the pseudo-conformal inversion, it is naturally related to the Cauchy problem for initial data randomized in the Fourier space, for which there is an extensive literature, cf.~\cite{Bourgain,B-T,N-O-R-S,L-M,B-O-P,S-X}, aiming at solutions with less regularity than required by the scaling or the deterministic argument. 
Concerning the initial data problem, almost sure scattering results have been obtained in \cite{D-L-M,D-L-M-2018,K-M-V} for the energy-critical equations (hence in the mass supercritical case) by randomization in the Fourier space. 

The idea behind those works, particularly in \cite{B-T} and the succeeding ones, is first to derive some better properties of the free solution almost surely than the original one before the randomization, and then to consider nonlinear perturbation in a deterministic way exploiting the better properties. 
Murphy \cite{Murphy} proved that the randomized free solutions have better global dispersive properties in the sense of space-time integrability due to the physical randomization, and then the unique existence of a scattering solution for $p\in(p_0(d),4/d)$ by Kato's argument \cite{Ka} using non-admissible Strichartz norms. 

Using a time-weighted Strichartz estimate instead of the non-admissible ones, we can extend Murphy's result to lower powers $p\in(p_1(d),4/d)$, where 
\EQ{
 p_1(d):=\frac{4-d+\sqrt{d^2+24d+16}}{4d} \in (2/d,p_0(d)).}

Since $p_1(3)<1=p_0(3)$ in particular, the above extension allows us to include systems with quadratic polynomial nonlinearity in three space dimensions, which appear in some physical models. 
This is why we consider the system \eqref{eq:NLS}. 
In the deterministic scattering problems in weighted spaces such as $\F H^s$, one often needs precise information about the nonlinear oscillation and resonance, which makes it highly non-trivial to extend results for the scalar equation \eqref{eq:power-NLS} to the system. 
In our case, however, such a fine analysis is not needed in using only a time weight, except for a conservation law of $L^2$-type quantity ensuring global existence of the solution. 
A similar idea of time-weighted norms was used in \cite{G-N-T-07} for the deterministic final-data problem of the Gross-Pitaevskii equation in three space dimension. 

Before stating our main result, we introduce randomization of the final-data to our system \eqref{eq:NLS},  following Murphy \cite{Murphy}. 

\begin{defi}[$L^{2}$-randomization]
\label{def:random-L2}
Let $\chi\in C_{c}^{\infty}(\R^{d})$ be a non-negative bump function with
\begin{align}
\chi(x)
=\begin{cases}
1
&(|x|<1),\\
0
&(|x|>2).
\end{cases}
\end{align}
We define a partition of unity $\{\chi_{k}\}_{k\in\Z^{d}}$ as
\begin{align}
\chi_{k}(x):=\frac{\chi(x-k)}{\sum\nolimits_{\ell\in\Z^{d}}\chi(x-\ell)}.
\end{align}

For each $k\in\Z^{d}$, and $a$, $b\in\{1,\dots,N\}$,
let $g_{a,b}^{k}$ be a mean-zero real-valued random variable with distribution $\mu_{a,b}^{k}$
on a probability space $(\Omega,\mathcal{A},\mathbb{P})$.
Moreover,
for all $a,b\in\{1,\dots,N\}$,
we assume that $\{g_{a,b}^{k}\}_{k\in\Z^d}$ is independent,
and that there exists $C>0$ such that
for all $\gamma\in\R$ and $k\in\Z^{d}$,
it holds that
\begin{align}
\label{ineq:g-condition}
\int_{\R}e^{\gamma x}\,d\mu_{a,b}^{k}(x)
\le e^{C\gamma^{2}}.
\end{align}
For example, we can take a mean-zero Gaussian random variable with a bounded variance $\sigma^k_{a,b}>0$, then $\mu_{a,b}^k(x)=(2\pi\sigma^k_{a,b})^{-1/2}\exp(-\frac{x^2}{2\sigma^k_{a,b}})$ and the above left side is equal to $\exp(\sigma^k_{a,b}\gamma^2/2)$. Another example is the case where the distribution $\mu^k_{a,b}$ is compactly supported. 

Then, for any $\m{u}\in (L^{2}(\R^{d}))^{N}$,
we define its randomization $\m{u}^{\omega}$ as
\EQ{ \label{def rand}
&
\m{u}^{\omega}(x):=\sum_{k\in\Z^{d}}\chi_{k}(x)G_{k}(\omega)\m{u}(x),
\pq
G_{k}(\omega)
:=\begin{bmatrix}
g_{1,1}^{k}(\omega) & \dots & g_{1,N}^{k}(\omega)\\
\vdots & \ddots & \vdots\\
g_{N,1}^{k}(\omega) & \dots &g_{N,N}^{k}(\omega)
\end{bmatrix}.}
\end{defi}

\bigskip

Next we specify the conditions needed for the nonlinearity and $M$.  
\begin{ass}
\label{ass:1}
Let $d,N\in\N$ and $p\in(p_{1}(d),4/d)$. 
We assume that the function $\m{f}:\C^N\to\C^N$ and the matrix $M=\diag(M_1,\dots,M_N)$ satisfy the following (i) and (ii):
\begin{enumerate}
\item[(i)]
$\m{f}(0)=0$ and there exists $C>0$ such that for all $\m{u},\m{v}\in\R^{d}$
\begin{align}
& |\m{f}(\m{u})-\m{f}(\m{v})|
\le C\max_{\m{w}\in\{\m{u},\m{v}\}}|\m{w}|^{p}|\m{u}-\m{v}|.
\end{align}
\item[(ii)]
There exists a positive-definite $N\times N$ Hermite matrix
$\Lambda$ 
satisfying $M\Lambda=\Lambda M$ and 
$\Im(\m{u}^*\Lambda \m{f}(\m{u}))=0$ for all $\m{u}\in\R^{d}$, where $\m{u}^*$ denotes the transposed complex conjugate. 
\end{enumerate}
\end{ass}

Now we are ready to state 
\begin{thm}[Main Theorem]
\label{thm:main}
Under the Assumption \ref{ass:1}, for all $\m{u}_{+}\in(L^{2}(\R^{d}))^{N}$
and for almost every $\omega\in\Omega$,
there exists a unique global solution 
$\m{u}\in C_t([0,\I);L^2(\R^{d}))$ to \eqref{eq:NLS} such that 
\EQ{ \label{Str diff}
 \|t^{\epsilon}(\m{u}-U(t)\m{u}_+^\omega)\|_{L^{q}_{t}L^{r}_{x}((0,\infty)\times\R^{d})}<\infty}
for all $q,r\in[2,\I]$ satisfying
\EQ{ \label{adm qr}
  \frac{1}{q}+\frac{d}{2r}=\frac d4, \pq (d,q,r)\not=(2,2,\I),}
and all $\e>0$ satisfying
\EQ{ \label{rang e}
 \frac{1}{p}-\frac{d}{4} <\epsilon < \frac{dp}{4(p+1)}.}
In particular, choosing $(q,r)=(\I,2)$ in \eqref{Str diff} implies that $\m{u}$ scatters in $L^2(\R^d)$ with the final-state $\m{u}_+^\omega$. 
The uniqueness holds under the condition \eqref{Str diff} for a fixed $(q,r,\e)$ satisfying the above and 
\EQ{ \label{dual qr}
 \max(1,2d/(d+2)) \le \frac{r}{p+1} \le 2, \pq (d,\frac{r}{p+1})\not=(2,1).}
\end{thm}

\begin{rem}
\label{rem:p1}
We have $p_{1}(d)<p_{0}(d)$ for all $d\in\N$,
so the above result is a small improvement of \cite{Murphy}.
In particular, it includes the case of $p=p_0(3)=1$ for $d=3$. 
For example, we can apply Theorem \ref{thm:main} to
the following system of nonlinear Schr\"odinger equations with quadratic interaction
\begin{align}
\begin{cases}
i\partial_{t}u+\frac{1}{2m_{1}}\Delta{u}=\lambda\bar{u}v,\\
i\partial_{t}v+\frac{1}{2m_{2}}\Delta{v}=\mu u^{2},
\end{cases}
\label{2-system}
\end{align}
where $(u,v):\R\times\R^{d}_{x}\to\C^{2}$ are the unknown functions, $m_{1},m_{2}\in\R\setminus\{0\}$ and $\lambda,\mu\in\C$ are some constants satisfying $\lambda\mu>0$. 
Assumption \ref{ass:1} holds with $\Lambda=\diag(|\mu|,|\lambda|)$. 
The deterministic scattering problems have been studied for this system in weighted spaces, see \cite{H-L-O,H-O-T}. 
\end{rem}

\begin{rem}
Those $(q,r)\in[2,\I]^2$ satisfying \eqref{adm qr} are called admissible pairs for the Strichartz estimate of the Schr\"odinger equation:  
\EQ{
 \|U(t)\fy\|_{L^q_tL^r_x(\R^{1+d})} \lesssim \|\fy\|_{L^2(\R^d)}.}
The range \eqref{rang e} of $\e$ is non-empty for positive $p$ iff $p>p_1(d)$. 
There is $(q,r)\in[2,\I]^2$ satisfying both \eqref{adm qr} and \eqref{dual qr} as long as $p\ge 0$ and $(d-2)p\le 4$, namely in the $H^1$ subcritical and critical cases. 

If $p>\frac{4}{d+2}$, then $(q,r,\e)$ can be chosen such that $q \e<1$ and $r<\I$, in which case we have $\|t^\e U(t)\m{u}^\omega_+\|_{L^q_tL^r_x((0,\I)\times\R^d)}<\I$ almost surely, so that the space-time condition for uniqueness may be simplified to $\|t^\e \m{u}\|_{L^q_tL^r_x((0,\I)\times\R^d)}<\I$. 
Note that $p=\frac{4}{d+2}$ is the critical exponent by scaling for the deterministic scattering in $\F H^1$.  
We have $p_1(d)>\frac{4}{d+2}$ for $d\le 5$, $p_1(d)=\frac{4}{d+2}$ for $d=6$, and $p_1(d)<\frac{4}{d+2}$ for $d\ge 7$, whereas $p_0(d)>\frac{4}{d+2}$ for all $d\in\N$. 
\end{rem}

A similar argument applies to the scattering problem around plane wave solutions to the nonlinear Schr\"odinger equation with the defocusing cubic power in three dimensions:
\EQ{
 i\p_t\phi + \Delta \phi = |\phi|^2\phi, \pq \phi(t,x):\R^{1+3}\to\C.}
By the gauge, the scaling and the Galilean invariance, the problem for any plane wave $|\phi|=const.>0$ is reduced to the simplest case $\phi=e^{-it}$. 
Then the change of variable $\psi=e^{-it}\phi$ transforms the problem to the scattering around the non-zero equilibrium $\psi=1$ to
\begin{align}
&
i\partial_{t}\psi+\Delta{\psi}=(|\psi|^{2}-1)\psi, \pq\psi(t,x):\R^{1+3}\to\C,
\tag{GP}\label{eq:GP}
\end{align}
which is sometimes called the Gross-Pitaevskii equation, in distinction to the standard NLS setting. 
A natural space for solutions is the nonlinear metric space of functions with finite energy:
\EQ{
 E(\psi):=\int_{\R^3} \frac{|\na\psi|^2}{2}+\frac{(|\psi|^2-1)^2}{4}dx<\I.}
The global well-posedness of \eqref{eq:GP} in the energy space was proven by G\'erard \cite{Gerard}. 
Gustafson, Tsai and the first author \cite{G-N-T-09} proved existence of a solution to any asymptotic profile in the energy space. 
It is worth noting that the asymptotic behavior of $\psi-1$ contains quadratic correction terms from the linearized evolution. 

Since the proof in \cite{G-N-T-09} is by a compactness argument similar to \cite{Nakanishi}, the uniqueness is an open question. 
Randomizing the final-data in the energy space, however, we can prove almost sure unique existence of a solution with the prescribed asymptotic behavior and some space-time integrability. 
Actually, we obtain the asymptotic convergence in a stronger topology than in the deterministic case \cite{G-N-T-09}, which is due to a low-frequency improvement by the randomization. 
See Section \ref{sec:GP} (and Theorem \ref{thm:main2}) for the detail.

We conclude this section with some notation used throughout the paper. 
$A\lesssim B$ denotes $A\le CB$ for some constants $C>0$,
and $A\sim B$ means that we have $A\lesssim B$ and $B\lesssim A$.
For $p\in[1,\I]$,
$p'=p/(p-1)$ denotes the H\"older conjugate exponent.
For $T>0$, $I_{T}$ denotes the interval $(T,\infty)$.
$\ell^p$ denotes the $L^p$ space with the counting measure. 
The mixed $L^p$ norms are denoted by 
\EQ{
 \|u(t,x)\|_{L^p_t L^q_x(A\times B)}=\left\{\int_A\left(\int_B |u(t,x)|^qdx\right)^{p/q}dt\right\}^{1/p}}
for $p,q<\I$, with obvious modifications in the case of $p=\I$ or $q=\I$. 
The case with more than two variables is treated in the same way. 
For vector-valued function $\m{u}=(u_{1},\dots,u_{N})$, we use the same notation of norms as for scalar functions, meaning 
$\|\m{u}\|_X
:=\{\sum\nolimits_{n=1}^{N}\|u_{n}\|_X^{2}\}^{1/2}$. 
The norms for intersection and sum are defined as usual
\EQ{
 \|u\|_{A\cap B}:=\max(\|u\|_A,\|u\|_B), \pq \|u\|_{A+B}:=\inf_{u=v+w}(\|v\|_A+\|w\|_B).}
Finally, for $\e\in\R$, $T>0$ and $p,q\in[1,\I]$, a time-weighted norm is denoted by
\EQ{
 \|u\|_{X^{p,q}_\e(I_T)} = \|t^\e u(t,x)\|_{L^p_t L^q_x(I_T\times\R^d)}.}

\bigskip

\noindent
{\bf Acknowledgements.}
This work was supported by JSPS KAKENHI Grant Number JP17H02854.

%%%%%%%%%%%%%%%%%%%%%%%%%%%%%%%%%%%
%%%%%%%%%%%%%%%%%%%%%%%%%%%%%%%%%%%
%%%%%%%%%%%%%%%%%%%%%%%%%%%%%%%%%%%

\section{Linear estimates} \label{sec:propositions}
In this section, we prepare basic linear estimates, some are probabilistic and others are deterministic. 

\begin{lem}[Large Deviation Estimate]
\label{cor:LDE}
Let $d,N\in\N$ and $\{G_{k}\}_{k\in\Z^d}$ be as in Definition \ref{def:random-L2}.
Then there is a constant $C>0$ such that for any $\alpha\in[2,\infty)$ and any $\{\m{c}_{k}\}_{k\in\Z^{d}}\in(\ell^{2}_{k}(\Z^{d};\C))^{N}$, it holds that 
\begin{align}
\nor{\sum\nolimits_{k\in\Z^{d}}G_{k}\m{c}_{k}}_{L^{\alpha}_{\omega}(\Omega)}
\le C\ro{\alpha}\nor{\m{c}_{k}}_{\ell^{2}_{k}(\Z^{d})}.
\end{align}
\end{lem}
\begin{proof}
This is \cite[Lemma 3.1]{B-T} in the scalar case $N=1$. The extension to vectors is obvious. 
\end{proof}

\begin{lem} \label{lem:a.s.linear0}
Let $d,N\in\N$ and $\{G_{k}\}_{k\in\Z^d}$ be as in Definition \ref{def:random-L2}, and 
let $U(t)$ be as in \eqref{def Ut} with some $M_1,\dots,M_N\in\R\setminus\{0\}$. 
Then there exists a constant $C>0$ such that for any $a\in(0,\I)$, $b\in(2,\I)$ and $\e\ge 0$ satisfying 
\EQ{ \label{cond eps}
 \epsilon_{0} :=-\epsilon+\bra{\frac{d}{2}-\frac{d}{b}}-\frac{1}{a}>0,}
and for any $\m{\fy}\in (L^2(\R^d))^{N}$, $\alpha\in[\max(a,b),\I)$ and $T>0$, we have 
\EQ{
 \|U(t)\m{\fy}^\omega\|_{L^\alpha_\omega(\Omega; X_\e^{a,b}(I_T))} \le C\sqrt\alpha \e_0^{-1/a}T^{-\e_0}\|\m{\fy}\|_{L^2(\R^d)},}
where $\m{\fy}^\omega$ is the randomization of $\m{\fy}$ defined by \eqref{def rand}. 
\end{lem}
\begin{proof}
First note that by explicit integration we have for any $p\in(0,\I)$, $\e>0$, and $T>0$, 
\EQ{ \label{t power norm}
 \|t^{-\e}\|_{L^p_t(I_T)} = (\e p-1)^{-1/p} T^{-\e+1/p} \sim (\e-1/p)^{-1/p} T^{-(\e-1/p)}.}

If $a<2$, then let $a_1:=2$ and $\e_1:=\e+1/a-1/2+\e_0/2$. We have an embedding $X^{a_1,b}_{\e_1}(I_T)\subset X^{a,b}_\e(I_T)$ by H\"older: putting $1/a_2:=1/a-1/2$,
\EQ{ \label{emb X}
 \|\m{\fy}\|_{X^{a,b}_\e(I_T)} \le \|t^{-(\e_1-\e)}\|_{L^{a_2}_t(I_T)}\|\m{\fy}\|_{X^{a_1,b}_{\e_1}(I_T)} \lec \e_0^{1/2-1/a}T^{-\e_0/2}\|\m{\fy}\|_{X^{a_1,b}_{\e_1}(I_T)},}
since $\e_1-\e=1/a_2+\e_0/2>1/a_2$. 
If $a\ge 2$, then let $a_1:=a$ and $\e_1:=\e$. 

In both cases, using Minkowski's inequality,
Lemma \ref{cor:LDE}, 
and the dispersive decay for the Schr\"odinger equation, we obtain
\EQ{
&\label{ineq:a.s.linear0}
\|U(t)\m{\varphi}^{\omega}\|
	_{L^{\alpha}_{\omega}(\Omega;X^{a_1,b}_{\e_1}(I_{T}))}
\\&\le\nor{\|U(t)\m{\varphi}^{\omega}\|
	_{L^{\alpha}_{\omega}(\Omega)}}_{X^{a_1,b}_{\e_1}(I_{T})}
=\nor{\nor{\textstyle\sum_{k\in\Z^{d}}G_{k}[U(t)\psi_{k}\m{\varphi}]}
	_{L^{\alpha}_{\omega}(\Omega)}}_{X^{a_1,b}_{\e_1}(I_{T})}
\\&
\lesssim\sqrt{\alpha}\nor{\nor{U(t)\psi_{k}\m{\varphi}}
	_{\ell^{2}_{k}(\Z^{d})}}_{X^{a_1,b}_{\e_1}(I_{T})}
\le\sqrt{\alpha}\nor{\nor{U(t)\psi_{k}\m{\varphi}}
	_{X^{a_1,b}_{\e_1}(I_{T})}}_{\ell^{2}_{k}(\Z^{d})}
\\&
\lesssim\sqrt{\alpha}\nor{t^{\e_1-d/2+d/b}
	\nor{\psi_{k}\m{\varphi}}_{L^{b'}_{x}(\R^{d})}}_{\ell^{2}_{k}L^{a_1}_t(\Z^{d}\times I_T)}
\\& \lesssim\sqrt{\alpha}\nor{t^{\e_1-d/2+d/b}}_{L^{a_1}_{t}(I_{T})}
	\|\m{\varphi}\|_{L^{2}_{x}(\R^{d})},
}
where we used that $|\supp\psi_k|=|\supp\psi_0|\lesssim 1$ and $\sum_{k\in\Z^d}|\psi_k(x)|^2\sim 1$ in the last step. If $a\ge 2$, then $\e_1-d/2+d/b=-1/a-\e_0$, so 
\EQ{
 \|t^{\e_1-d/2+d/b}\|_{L^{a_1}_t(I_T)} \sim \e_0^{-1/a} T^{-\e_0},}
implies the desired estimate. If $a<2$, then $\e_1-d/2+d/b=-1/a-\e_0-\e+\e_1=-1/a_1-\e_0/2$, so
\EQ{
 \|t^{\e_1-d/2+d/b}\|_{L^{a_1}_t(I_T)} \sim \e_0^{-1/a_1} T^{-\e_0/2},}
together with \eqref{emb X} implies the desired estimate. 
\end{proof}

\begin{lem}
\label{prop:t-Strichartz}
Let $d,N\in\N$ and 
let $U(t)$ be as in \eqref{def Ut} with some $M_1,\dots,M_N\in\R\setminus\{0\}$. Let $(q_{0},r_{0})$, $(q_{1},r_{1})\in[2,\I]^2$ be admissible pairs. 
Then there exists a constant $C>0$ such that for any $\e\ge 0$, $T>0$ 
and any $\m{F}\in (X^{q_{1}',r_{1}'}_{\epsilon}(I_T))^{N}$, we have
\begin{align}
&\label{ineq:t-Strichartz}
\nor{\int_{\I}^{t}U(t-s)\m{F}(s)\,ds}_{X^{q_{0},r_{0}}_{\epsilon}(I_T)}
\le C \|\m{F}\|_{X^{q_{1}',r_{1}'}_{\epsilon}(I_T)}.
\end{align}
\end{lem}
\begin{proof}
In the case without weight $\e=0$, this is the well-known Strichartz estimate for the Schr\"odinger equation, to which the case of $\e>0$ is reduced as follows.
We interpret $I_{T}$ as the indicator function of $I_{T}$.
For any $\e>0$ and $1\le q<\I$, let $g(t):=(\e q)^{1/q}t^{\e-1/q}$. Then for $t>T>0$, we have 
\EQ{
 t^{\e q} = T^{\e q} + \int_T^t \e q  \tau^{\e q-1}d\tau = T^{\e q} + \|g(\tau)I_\tau(t)\|_{L^q_\tau(I_T)}^q.}
Hence for any measurable $u:I_T\to\R$ we obtain, using Fubini and H\"older for $\ell^q(\{1,2\})$ as well, 
\EQ{
 \|t^\e u(t)\|_{L^q_t(I_T)} &\le \|T^\e u(t)\|_{L^q_t(I_T)} + \|g(\tau)I_\tau(t)u(t)\|_{L^q_tL^q_\tau(I_T\times I_T)}
 \\&= T^\e \|u\|_{L^q_{t}(I_T)} + \|g(\tau)\|u(t)\|_{L^q_t(I_\tau)}\|_{L^q_\tau(I_T)}
 \le 2\|t^\e u(t)\|_{L^q_t(I_T)}.}
If $\e>0$ and $q_0<\I$, then using the above estimate and the Strichartz without weight, we obtain 
\EQ{
 &\nor{\int_{\I}^{t}U(t-s)\m{F}(s)\,ds}_{X^{q_{0},r_{0}}_{\epsilon}(I_T)}
 \\& \lesssim T^\e \|\m{F}\|_{X^{q_1',r_1'}_0(I_T)} + 
    \nor{g(\tau)\|\m{F}\|_{X^{q_1',r_1'}_0(I_\tau)}}_{L^{q_0}_\tau(I_T)}
 \\&\le \|\m{F}\|_{X^{q_1',r_1'}_\e(I_T)} +
  \|g(\tau)\m{F}(t,x)I_\tau(t)\|_{L^{q_1'}_t L^{q_0}_\tau L^{r_1'}_x(I_T \times I_T\times \R^d)}
\\& \le 2 \|\m{F}\|_{X^{q_1',r_1'}_\e(I_T)},}
where we used $L^{q_1'}_tL^{q_0}_\tau\subset L^{q_0}_\tau L^{q_1'}_t$ by Minkowski and $q_0\ge q_1'$ in the second inequality. 
The case of $q_0=\I$ is immediate from the Strichartz estimate without weight into $L^2_x$. 
\end{proof}

%%%%%%%%%%%%%%%%%%%%%%%%%%%%%%%%%%%
%%%%%%%%%%%%%%%%%%%%%%%%%%%%%%%%%%%
%%%%%%%%%%%%%%%%%%%%%%%%%%%%%%%%%%%

\section{Proof of the main theorem}
\label{sec:proof}

The proof of Theorem \ref{thm:main} consists of three steps. 
First, for the free solution of the randomized final-data $\m{u}^0:=U(t)\m{u}_+^\omega$, we have dispersive decay estimates, which are almost the best possible, by Lemma \ref{lem:a.s.linear0}. 
Then using this decay property, we can construct a unique solution $\m{u}$ asymptotic to $\m{u}^0$ locally around time infinity in a time-weighted Strichartz space, under the Assumption \ref{ass:1} (i). 
Finally, the local solution is extended to a global one using Assumption \ref{ass:1} (ii). 

The unique local solution is given by the following lemma, which is a deterministic statement. 
\begin{lem}
\label{lem:contraction}
Let $d,N\in\N$, $p\in(p_{1}(d),\frac{4}{d})$, and $\m{f}:\C^N\to\C^N$ satisfy Assumption \ref{ass:1} (i). 
Let $\e>0$ satisfy \eqref{rang e}. 
Let $(q_0,r_0)\in(0,\I]^2$ and $(q_1,r_1)\in[2,\I]^2$ such that $(q,r)=(q_j,r_j)$ satisfies \eqref{adm qr} and \eqref{dual qr} for $j=0,1$. 
Then there exists $\eta_{0}>0$ such that the following holds. 
For any $\eta\in(0,\eta_{0})$, $\m{u}_{+}\in (L^{2}(\R^{d}))^{N}$ and $T\ge 1$ satisfying 
\EQ{ \label{dec free}
 \nor{U(t)\m{u}_{+}}_{X^{q_0,r_0}_{\epsilon}(I_{T})}\le \eta,}
there exists a unique local solution
$\m{u}\in C_{t}(I_T;(L^2_x(\R^d))^N)$ to \eqref{eq:NLS} on $I_{T}$ such that 
\EQ{ \label{diff est}
 \|\m{u}-U(t)\m{u}_{+}\|_{X^{\I,2}_\e(I_T)\cap X^{q_1,r_1}_{\epsilon}(I_{T})}\le \eta.}
\end{lem}
\begin{rem}
The above solution $\m{u}$ scatters in $L^2(\R^d)$ with the final-data $\m{u}_+$ because of the $X^{\I,2}_\e(I_T)$ estimate in \eqref{diff est}. 
The pair $(q_1,r_1)$ is admissible, but $(q_0,r_0)$ is not necessarily so, since $q_0$ can be in $(0,2)$, or $r_0$ can be in $(\frac{2d}{d-2},\I]$, for $d\ge 3$. This extension of the range of $(q_0,r_0)$ is needed to satisfy \eqref{dec free} when $p\in(p_1(d),\frac{4}{d+2})$, which is not empty for $d\ge 7$. 
The small constant $\eta_0$ can be taken uniformly for $[q_1,r_1]$ except for the limit $(q_1,r_1)\to(2,\I)$ in $d=2$, where the Strichartz estimate blows up. 
\end{rem}
\begin{proof}
Let $\m u^0:=U(t)\m{u}_{+}$ and $\m u^1:=\m u-\m u^0$. 
Then \eqref{eq:NLS} for $\m u$ with the final-data $\m{u}_+$ is rewritten for $\m u^1$ as
\EQ{
 \m u^1 = i\int_t^\I U(t-s)\m{f}(\m u^0 + \m u^1)\,ds.}
Hence it suffices to show that the mapping $\Phi$ defined by
\EQ{
 \Phi(\m v):= i\int_t^\I U(t-s)\m{f}(\m u^0 + \m v)\,ds}
is a contraction on the closed $\eta$-ball of the Banach space 
\EQ{
 X_T=\bra{C_{t}(I_{T};L^2(\R^{d}))\cap X^{\I,2}_\e(I_T)\cap X^{q_1,r_1}_\e(I_{T})}^N,}
if $\eta>0$ is small enough. 

For $j=0,1$, let $\ti q_j:=q_j/(p+1)$, $\ti r_j:=r_j/(p+1)$, and 
\EQ{ 
 X^j_T:=X^{q_j,r_j}_\e(I_T), \pq \ti X^j_T:=X_{(p+1)\e}^{\ti q_j,\ti r_j}(I_T).} 
Then for any $\m v\in X_T$ with $\|\m v\|_{X_T}\le\eta$, we have by Assumption \ref{ass:1} (i) and H\"older in $(t,x)$,  
\EQ{
 \|\m f(\m u^0+\m v)\|_{\ti X^0_T+\ti X^1_T}
& \lec \||\m u^0|^{p+1}\|_{\ti X^0_T} + \||\m v|^{p+1}\|_{\ti X^1_T}
\\& \lec \|\m u^0\|_{X^0_T}^{p+1} + \|\m v\|_{X^1_T}^{p+1} \lec \eta^{p+1}.}
Since \eqref{dual qr} holds for $r=r_j$, there is a unique $\hat q_j\in[2,\I]$ such that $(\hat q_j,\ti r_j')$ is admissible, and then 
\EQ{ \label{nonlin decay}
 \frac{1}{\hat q_j'}+\frac{d}{2\ti r_j}=\frac d4+1,\pq \frac{2}{\ti q_j}+\frac{d}{\ti r_j}=\frac d4(p+1)
 \implies \frac{1}{\hat q_j'}-\frac{1}{\ti q_j}=1-\frac{dp}{4}< p\e,}
where the last inequality follows from \eqref{rang e}. 
Hence by H\"older in $t$, 
\EQ{
 \|\m f(\m u^0+\m v)\|_{\hat X^0_T+\hat X^1_T} \lec \|t^{p\e}\|_{L^a_t(I_T)}\|\m f(\m u^0+\m v)\|_{\ti X^0_T+\ti X^1_T} \lec T^{1/a-p\e} \eta^{p+1},}
where $1/a:=1-dp/4$ and $\hat X^j_T:=X^{\hat q_j',\ti r_j}_\e(I_T)$. 
Since $T\ge 1$, the right hand side can be made much smaller than $\eta$ by choosing $\eta_0>0$ small enough. 
On the other hand, the weighted Strichartz Lemma \ref{prop:t-Strichartz} implies 
\EQ{
 \|\Phi(\m v)\|_{X_T} \lec \|\m f(\m u^0+\m v)\|_{\hat X^0_T+\hat X^1_T}.}
Therefore, it holds that $\|\Phi(\m v)\|_{X_T}\le \eta$, if $\eta_{0}>0$ is sufficiently small.

Similarly, if $\m v^0,\m v^1\in X_T$ satisfy $\|\m v^j\|_{X_T}\le\eta$, then 
\EQ{
 \|\Phi(\m v^0)-\Phi(\m v^1)\|_{X_T}
 &\lec \|\m f(\m u^0+\m v^0)-\m f(\m u^0-\m v^1)\|_{\hat X^0_T+\hat X^1_T}
 \\&\lec T^{1/a-p\e} \|(|\m u^0|+|\m v^0|+|\m v^1|)^p|\m v^0-\m v^1|\|_{\ti X^0_T+\ti X^1_T}
 \\&\lec \bra{\|\m u^0\|_{X^0_T}^p+\|\m v^0\|_{X^1_T}^p+\|\m v^1\|_{X^1_T}^p}\|\m v^0-\m v^1\|_{X^1_T}
 \\&\lec \eta^p\|\m v^0-\m v^1\|_{X^1_T},}
where in the third inequality we used the interpolation 
\EQ{
 \||\m u^0|^p|\m v^0-\m v^1|\|_{\ti X^0_T+\ti X^1_T}
& \lec \||\m u^0|^p|\m v^0-\m v^1|\|_{[\ti X^0_T,\ti X^1_T]_{1/(p+1)}}
\\& \lec \|\m u^0\|_{X^0_T}^p\|\m v^0-\m v^1\|_{X^1_T},}
where $[\ti X^0_T,\ti X^1_T]_{1/(p+1)}=X^{q_2,r_2}_{(p+1)\e}(I_T)$ with 
\EQ{
  \bra{\frac1{q_2},\frac1{r_2}}:=p\bra{\frac1{q_0},\frac1{r_0}}+\bra{\frac1{q_1},\frac1{r_1}}.}
Therefore, if $\eta_0>0$ is small enough, then $\Phi$ is a contraction mapping on the closed $\eta$-ball of $X_T$, so the conclusion follows from the Banach fixed point theorem. 
\end{proof}

By Lemmas \ref{lem:a.s.linear0} and \ref{lem:contraction}, we can prove Theorem \ref{thm:main}.
\begin{proof}[Proof of Theorem \ref{thm:main}]
Take any $\e>0$ in \eqref{rang e}, and any $(q_1,r_1)\in[2,\I]^2$ satisfying \eqref{adm qr} and \eqref{dual qr}. 
In order to apply Lemma \ref{lem:contraction} to the randomized final-data $\m u_+^\omega$, we need some $(q_0,r_0)\in(0,\I]^2$ satisfying \eqref{adm qr}, \eqref{dual qr}, and 
\EQ{
 \|U(t)\m u_+^\omega\|_{X^{q_0,r_0}_\e(I_T)} <\eta_0}
for some $T\ge 1$ almost surely. In order to use Lemma \ref{lem:a.s.linear0}, we need 
\EQ{
 \e_0:=-\e+d(1/2-1/r_0)-1/q_0=1/q_0-\e>0.}
Among those $(q_0,r_0)$ satisfying \eqref{adm qr} and \eqref{dual qr}, the maximal $1/q_0$ is achieved when $r_0$ is the maximum in \eqref{dual qr}, namely $r_0=2(p+1)\in(2,\I)$, and by \eqref{adm qr}, 
\EQ{
 \frac1{q_0}=\frac{d}{4}-\frac{d}{2r_0}=\frac{dp}{4(p+1)}.}
Then \eqref{rang e} implies that $\e_0>0$, and so Lemma \ref{lem:a.s.linear0} implies 
\EQ{
 \|U(t)\m u_+^\omega\|_{L^\alpha_\omega(\Omega; X^{q_0,r_0}_\e(I_T))} \lec \sqrt\alpha \e_0^{-1/q_0}T^{-\e_0}\|\m u_+\|_{L^2(\R^d)}<\I.}

In particular, for almost every $\omega\in\Omega$, we have $U(t)\m u_+^\omega\in X^{q_0,r_0}_\e(I_1)$. 
The exceptional null set of $\omega$ can be chosen independent of $(q_0,r_0,\e)$, since there is a countable set of triples $(q_0,r_0,\e)$ such that the intersection of $X^{q_0,r_0}_\e(I_1)$ for those triples is included in that of any other triple satisfying the conditions. 

Since $q_0<\I$, the dominated convergence theorem implies that $\|U(t)\m u_+^\omega\|_{X^{q_0,r_0}_\e(I_T)}<\eta_0$ for sufficiently large $T$, and thus by Lemma \ref{lem:contraction}, there exists a unique local solution $\m{u}$ to \eqref{eq:NLS} on $I_T$ satisfying \eqref{diff est}. 
Applying the weighted Strichartz estimate once again as in the estimate on $\Phi(\m v)$ yields
\EQ{
 \|\m u-U(t)\m u_+^\omega\|_{X^{q,r}_\e(I_T)}<\I}
for all admissible $(q,r)$. This solution $\m u$ is extended uniquely to a global one satisfying  
\EQ{
 \m u \in (C(\R;L^2(\R^d)) \cap L^q_{loc}(\R;L^r(\R^d)))^N,}
by the global well-posedness in $L^2(\R^d)$ in the subcritical case $p<4/d$, see \cite{Tsutsumi-1987} for the case of \eqref{eq:power-NLS}. In the case of \eqref{eq:NLS}, the $L^2$ conservation is replaced with 
\EQ{
 \p_t\LR{\m u|\Lambda \m u}=2\LR{-M\Delta\m u + \m f(\m u) |i\Lambda u}=0,}
by Assumption \ref{ass:1} (ii), where $\LR{\m u|\m v}:=\Re\sum_{j=1}^N\int_{\R^d} u_j(x)\bar v_j(x)dx$ is the real inner product in $(L^2(\R^d))^N$. Since $\Lambda$ is positive definite, the above conservation law implies a priori bound of solutions in $L^2(\R^d)$, thereby the global well-posedness of \eqref{eq:NLS}. 
Thus for any $\e$ in \eqref{rang e}, we obtain a global solution $\m u$ satisfying \eqref{Str diff} for all $(q,r)$ in \eqref{adm qr}. It remains to show the uniqueness among all such solutions. 

Let $\m{u},\m{u}'\in(C_{t}(\R;L^{2}_{x}(\R^{d}))^N$ be two solutions
to \eqref{eq:NLS} satisfying \eqref{Str diff} for some admissible $(q_1,r_1)$ with \eqref{dual qr}, and some $\e$ in \eqref{rang e}. 
Then, the dominated convergence theorem implies that for any $\eta>0$ and for sufficiently large $T>1$, \eqref{diff est} holds both for $\m u$ and for $\m u'$. 
Then the uniqueness in Lemma \ref{lem:contraction} implies that $\m u=\m u'$ on $I_T$, and the well-posedness implies that $\m u=\m u'$ on the whole $\R$, which concludes Theorem \ref{thm:main}.
\end{proof}

%%%%%%%%%%%%%%%%%%%%%%%%%%%%%%%%%%%
%%%%%%%%%%%%%%%%%%%%%%%%%%%%%%%%%%%
%%%%%%%%%%%%%%%%%%%%%%%%%%%%%%%%%%%

\section{Application to the Gross-Pitaevskii Equation}
\label{sec:GP}

In this section, we study the final-data problem in the energy space for the Gross-Pitaevskii equation \eqref{eq:GP} in three space dimension. It is natural to put 
\EQ{
 u:=\psi-1=u_1+iu_2\in\R\oplus i\R,}
then the equation \eqref{eq:GP} is rewritten as 
\EQ{ \label{eq:GP-u}
 i\dot u + \De u - 2u_1 =  3u_1^2+ u_2^2 + |u|^2u_1 + i(2u_1u_2+|u|^2u_2).}
Following \cite[Section 4]{G-N-T-09}, we further transform the unknown function $u$ by 
\EQ{
 & U:=\sqrt{-\De(2-\De)^{-1}},
 \pq M(u):=u_1+iUu_2 + (2-\De)^{-1}|u|^2,
 \\& \z:=U^{-1}M(u),}
where $u\mapsto u_1+iUu_2$ is the natural $\R$-linear transform making the linearized evolution into a unitary group, while the quadratic transform $M$ is a local homeomorphism from the energy space to $H^1(\R^3)$ for $L^6(\R^3)$-small functions, playing crucial roles in the nonlinear analysis, removing some singular terms around the zero frequency. In fact, the energy is rewritten as  
\EQ{
 E(\psi) = \int_{\R^3} \frac{|\na u|^2}{2}+\frac{||u|^2+2u_1|^2}{4}dx = \int_{\R^3} \frac{|\na\z|^2}{2}+\frac{|U|u|^2|^2}{4}dx,}
while the equation is transformed into
\EQ{ \label{eq:ze}
 & i\p_t \z-H\z = N(u):=2u_1^2+|u|^2u_1 -i H^{-1}\na\cdot\{4u_1\na u_2+\na(|u|^2u_2)\},
 \\& H:=\sqrt{-\De(2-\De)}.}
Also recall from \cite{G-N-T-06} that the propagator $e^{-itH}$ for the linearized equation of $\z$ enjoys the same dispersive as the Schr\"odinger $e^{it\De}$ (possibly excepting $L^\I_x$). 
Actually it is better than $e^{it\De}$ in the low frequency, or gains some power of $U$. 
\begin{lem}[\cite{G-N-T-06}]
\label{lem:H-U-P}
Let $f:\R^{3}_{x}\to\C$ be measurable.
For $r\in[2,\infty)$ we have
\begin{align}
&\label{ineq:H-est1}
\nor{e^{-itH}f}_{L^r(\R^{3})}
\lesssim |t|^{-3(\frac{1}{2}-\frac{1}{r})}\nor{U^{\frac{1}{2}-\frac{1}{r}}f}_{L^{r'}(\R^{3})}.
\end{align}
\end{lem}
The above implies the same Strichartz estimate for $e^{-itH}$ as for the Schr\"odinger equation, which can be weighted as in Lemma \ref{prop:t-Strichartz}:
\begin{lem}
\label{lem:t-Strichartz 4H}
There exists a constant $C>0$ such that for any admissible pairs $(q_{0},r_{0})$, $(q_{1},r_{1})\in[2,\I]^2$ on $\R^3$, $\e\ge 0$, $T>0$ and $F\in X^{q_{1}',r_{1}'}_{\epsilon}(I_T)$, we have
\begin{align}
 \nor{\int_{\I}^{t}e^{-i(t-s)H}F(s)\,ds}_{X^{q_{0},r_{0}}_{\epsilon}(I_T)}
 \le C \|F\|_{X^{q_{1}',r_{1}'}_{\epsilon}(I_T)}.
\end{align}
\end{lem}

It also implies that scattering solutions are vanishing in $L^6(\R^3)$, hence the energy space for the asymptotic profiles of $\z=U^{-1}M(u)$ is $U^{-1}H^1(\R^3)=\dot H^1(\R^3)$. 

For the final-data problem, existence of a solution $u$ of \eqref{eq:GP-u} such that $\|M(u)-e^{-iHt}z_+\|_{H^1(\R^3)}\to 0$ is given for any $z_+\in H^1(\R^3)$ by \cite[Theorem 1.2]{G-N-T-09}. 
In terms of $\z=U^{-1}M(u)$, it means the existence of a scattering solution $\z$ of \eqref{eq:ze} in $\dot H^1(\R^3)$. 
We consider a randomized version with uniqueness. 
Note that the physical randomization in Definition \ref{def:random-L2} does not commute with derivatives, unlike similar randomization in the Fourier side. If we directly apply that randomization to a final-state in Sobolev spaces, then estimates on the randomized data and the corresponding free solution get extra terms from differentiating the partition of unity $\chi_k$.  
Therefore, it seems more natural to apply the randomization after transforming the energy space onto $L^2(\R^3)$. 

\begin{defi}[$\dot H^1$-randomization]
\label{def:random-dotH1}
For $\varphi\in \dot H^1(\R^{3})$,
we define its randomization $\varphi^{\omega,1}$ by
\EQ{ \label{random-dotH1}
 \fy^{\omega,1}:=|\na|^{-1}(|\na|\varphi)^{\omega}=|\na|^{-1}\sum_{k\in\Z^d}G_k(\om)\chi_k |\na|\fy.}
\end{defi}

Then we obtain the following, which is a randomized version of \cite[Theorem 1.2]{G-N-T-09}.  
\begin{thm}
\label{thm:main2}
For any $\z_{+}\in \dot H^1(\R^{3})$ and for almost every $\omega\in\Omega$,
there exists a unique global solution $u\in C_{t}(\R;H^1(\R^3))$ to \eqref{eq:GP-u} such that $v:=U^{-1}u_1+iu_2$ satisfies 
\EQ{ \label{Str diff GP}
 \|t^\e\LR{\na}(v-e^{-itH}\z_+^{\om,1})\|_{L^q_tL^r_x(I_T\times\R^3)}<\I,}
for some $T>0$, all admissible pair $(q,r)$, and all $\e\in(1/4,3/8)$. 
The uniqueness holds under the condition \eqref{Str diff GP} for a fixed $(q,r,\e)$ with $r\in[12/5,4]$. Moreover,  
\EQ{
 \|v(t)-e^{-itH}\z_+^{\om,1}\|_{H^1(\R^3)} + \|\z(t)-e^{-itH}\z_+^{\om,1}\|_{H^1(\R^3)} = o(t^{-\e})}
as $t\to\I$ for all $\e<3/8$, namely the scattering holds for $v$, $\z$ and $M(u)=U\z$ in $H^1(\R^3)$. 
\end{thm}
The above conditions on $q,r,\e$ are the same as in Theorem \ref{thm:main} for $d=3$ and $p=1$. 
Even though $\z_{+}\in \dot H^1(\R^{3})$,
$z$ and $v$ scatter in $H^1(\R^{3})$ almost surely, and correspondingly, the solution $\psi$ belongs to the smaller space $1+H^1(\R^3)$, where the global well-posedness was proved by Bethuel and Saut \cite{B-S}. 
This is because the randomization by Murphy improves the low frequency. 
It makes the quadratic term in the transform $M$ asymptotically negligible as $t\to\I$ in $H^1(\R^3)$, which is in contrast to the deterministic case, cf.~\cite[Remark 4.1]{G-N-T-09}. 
Actually the following implies that $\z_+^{\om,1}\in\dot H^s$ for $-1/2<s\le 1$ almost surely. 
\begin{lem}
\label{prop:random-L2-range}
Let $d\in\N$, $\phi\in L^{2}(\R^{d})$. 
Then we have $\phi^{\omega}\in\dot{H}^{s}(\R^{d})$ almost surely for all $s\in(-d/2,0]$. 
If $\phi\in \dot H^1(\R^d)$, then $\phi^{\omega,1}\in\dot H^s(\R^d)$ almost surely for all $s\in(1-d/2,1]$. 
\end{lem}
\begin{proof}
Let $\phi\in L^2(\R^d)$ and $s\in(-d/2,0]$. 
For all $\alpha\in[2,\infty)$, 
using Sobolev's inequality with $\frac{1}{r}:=\frac{1}{2}-\frac{s}{d}$,
it holds that
\EQ{
&
\nor{\phi^{\omega}}_{L^{\alpha}_{\omega}\dot{H}^{s}_{x}(\Omega\times\R^{d})}
\le\nor{\nor{\sum\nolimits_{k\in\Z^{d}}g_{k}(\omega)|\nabla|^{s}(\psi_{k}\phi)}
	_{L^{\alpha}_{\omega}(\Omega)}}_{L^{2}_{x}(\R^{d})}
\\&\lesssim\nor{\big\||\nabla|^{s}(\psi_{k}\phi)\big\|
	_{\ell^{2}_{k}(\Z^{d})}}_{L^{2}_{x}(\R^{d})}
 =\nor{\big\||\nabla|^{s}(\psi_{k}\phi)\big\|
	_{L^{2}_{x}(\R^{d})}}_{\ell^{2}_{k}(\Z^{d})}
\\&\lesssim\nor{\|\psi_{k}\phi\|_{L^{r}_{x}(\R^{d})}}_{\ell^{2}_{k}(\Z^{d})}
\lesssim\nor{\|\psi_{k}\phi\|_{L^{2}_{x}(\R^{d})}}_{\ell^{2}_{k}(\Z^{d})}
\sim\|\phi\|_{L^{2}_{x}(\R^{d})}.
}
Thus we deduce that $\phi^\omega\in\dot H^s$ for almost every $\om$ and for all $s\in(-d/2,0]$, taking a countable dense set of $s$ in $(-d/2,0]$ including $0$. 
Then the claim on $\phi^{\omega,1}$ follows from the fact that $|\na|$ is an isomorphism from $\dot H^{s+1}$ onto $\dot H^{s}$. 
\end{proof}

In particular, the reference free solution $e^{-itH}\z_+^{\om,1}\in L^\I(\R;H^1(\R^3))$ almost surely, and then it decays in $L^3(\R^3)$ because of the dispersive decay estimate \eqref{ineq:H-est1} together with the Sobolev embedding $H^1(\R^3)\subset L^3(\R^3)$. 
The $L^3$-smallness makes it easy to invert the quadratic transform $M$ as follows. 
\begin{lem} \label{lem:z-u}
There exists a constant $\eta_*>0$ with the following property. Let $B_*:=\{\fy\in H^1(\R^3)\mid \|\fy\|_{L^3(\R^3)}\le\eta_*\}$. Then for any $\fy\in B_*$, there exists a unique $u\in H^1(\R^3)$ satisfying 
\EQ{
 U\fy = M(u), \pq \|u\|_{L^3(\R^3)} \le 2\|U\fy_1+i\fy_2\|_{L^3(\R^3)}.}
Let $g:B_*\to H^1(\R^3)$ be the map defined by $g(\fy):=u$ given above. Then there exists a constant $C>0$ such that for any $r\in[2,6]$, $s\in[0,1]$ and $\fy,\psi\in B_*\cap H^s_r(\R^3)$, we have 
\EQ{
 \|g(\fy)-g(\psi)\|_{H^s_r(\R^3)} \le C\|\fy-\psi\|_{H^s_r(\R^3)}.}
\end{lem}
Note that $g(0)=0$ because $M(0)=0$, so $\|g(\psi)\|_{H^s_r}\le C\|\psi\|_{H^s_r}$. 
\begin{proof}
For $\fy,u\in H^1(\R^3)$, let $\ti\fy:=U\fy_1+i\fy_2$ and $\Phi_\fy(u):=\ti\fy+(2-\De)^{-1}|u|^2$. Then 
$U\fy=M(u)$ is equivalent to $u=\Phi_\fy(u)$. 
First we show that $\Phi_\fy$ is a contraction on 
\EQ{
 K_\eta:=\{u\in H^1(\R^3) \mid \|u\|_{L^3(\R^3)}\le 2\eta\},}
provided that $\|\ti\fy\|_{L^3(\R^3)}\le\eta$ and $\eta>0$ is small enough. 
For any $r\in(3/2,\I)$, and any $u\in H^1_r(\R^3)\cap L^3(\R^3)$, we have 
\EQ{
 \|(2-\De)^{-1}|u|^2\|_{H^1_r} \sim \||u|^2\|_{H^{-1}_r} \lec \||u|^2\|_{L^{\bar r}} \le \|u\|_{L^3}\|u\|_{L^r},}
where $\bar r\in(1,\I)$ is determined by $1/\bar r=1/r+1/3$. Restricting $r$ to the closed interval $[2,6]$ allows us to take the implicit constants independent of $r$. 
In particular, choosing $r=2,3$, we deduce that $\Phi_\fy$ maps $K_\eta$ into itself if $\eta>0$ is small enough. The difference is estimated in the same way as 
\EQ{
 \|\Phi_\fy(u)-\Phi_\fy(w)\|_{H^1_r} \lec \||u|^2-|w|^2\|_{L^{\bar r}} \le \|u+w\|_{L^3}\|u-w\|_{L^r}.}
Hence $\Phi_\fy$ is a contraction on $K_\eta$ for small $\eta>0$, and so there is a unique fixed point $u\in K_\eta$, which is the unique solution of $U\fy=M(u)$ in $K_\eta$. Since $\|\ti\fy\|_{L^3}\lec\|\fy\|_{L^3}\le\eta_*$, choosing $\eta_*>0$ small enough ensures the smallness of $\|\ti\fy\|_{L^3}\le\eta$. 
Thus the map $g:B_*\to H^1$ is well-defined and given by the iteration limit
\EQ{
 g(\fy)=\lim_{n\to\I}(\Phi_\fy)^n(0),}
in $K_\eta$. By the same estimate as above, we have for any $\fy,\psi\in B_*$ and any $u,w\in K_\eta$,
\EQ{
 \|\Phi_\fy(u)-\Phi_\psi(w)\|_{H^s_r} \le \|\ti\fy-\ti\psi\|_{H^s_r}+C\eta\|u-w\|_{L^r}.}
Choosing $\eta_*>0$ small ensures that $C\eta<1/2$ on the right side. Then by induction on $n\in\N$, we obtain  
\EQ{
 \|(\Phi_\fy)^n(0)-(\Phi_\psi)^n(0)\|_{H^s_r} \le 2\|\ti\fy-\ti\psi\|_{H^s_r}.}
and then sending $n\to\I$, 
\EQ{
 \|g(\fy)-g(\psi)\|_{H^s_r} \le 2\|\ti\fy-\ti\psi\|_{H^s_r} \lec \|\fy-\psi\|_{H^s_r},}
as claimed.
\end{proof}

By the above lemma, it suffices to solve the equation \eqref{eq:ze} coupled with $u=g(\z)$, namely
\EQ{ \label{eq:ze-g}
 i\p_t\z - H\z = N(g(\z))}
for large $t>0$, since we are looking for solutions $u$ in $H^1(\R^3)$ decaying in $L^3(\R^3)$ by scattering. 
In order to prove Theorem \ref{thm:main2},
we prepare Lemma \ref{lem:a.s.linear02} and Lemma \ref{lem:contraction2}
that correspond to Lemma \ref{lem:a.s.linear0}
and Lemma \ref{lem:contraction} for \eqref{eq:NLS}.

\begin{lem}
\label{lem:a.s.linear02}
Let $a\in(0,\I)$, $b\in(2,4)$ and $\e\ge 0$ satisfy \eqref{cond eps}. 
Then for any $T>0$, $\al\in[\max(a,b),\I)$ and $\varphi\in\dot{H}^{1}(\R^{3})$,
we have 
\begin{align}
 \|\LR{\na}e^{-itH}\fy^{\om,1}\|_{L^\al_\om(\Omega;X^{a,b}_\e(I_T))} \lec \sqrt\al \e_0^{-1/\al}T^{-\e_0}\|\varphi\|_{\dot{H}^{1}(\R^{3})},
\end{align}
where $\LR{\na}:=\sqrt{2-\De}$, $\fy^{\om,1}$ is defined in \eqref{random-dotH1}, and $\e_0$ is defined in \eqref{cond eps}. 
\end{lem}
\begin{proof}
If $a<2$, then let $a_1:=2$ and $\e_1:=\e+1/a-1/2+\e_0/2$. Otherwise, let $a_1:=a$ and $\e_1:=\e$. 
By the same argument as in the proof of Lemma \ref{lem:a.s.linear0} using Minkowski's inequality and Lemma \ref{lem:H-U-P}, 
it holds that
\EQ{
&\nor{\LR{\na}e^{-itH}\varphi^{\omega,1}}_{L^{\alpha}_{\omega}(\Omega;X^{a_1,b}_{\e_1}(I_{T}))}
\\&=\nor{\textstyle\sum_{k\in\Z^{3}} G_{k}\sqb{e^{-itH}U^{-1}(\psi_{k}|\nabla|\varphi)}}_{L^{\alpha}_{\omega}(\Omega;X^{a_1,b}_{\e_1}(I_{T}))}
\\& 
\lec\sqrt{\alpha}\nor{\nor{e^{-itH}U^{-1}(\psi_{k}|\nabla|\varphi)}_{X^{a_1,b}_{\e_1}(I_{T})}}_{\ell^{2}_{k}(\Z^{3})}
\\&
\lesssim\sqrt{\alpha}\nor{t^{\epsilon_1-d/2-d/b}}_{L^{a_1}_{t}(I_{T})}
	\nor{\big\|U^{-1/2-1/b}\psi_{k}|\nabla|\varphi\big\|_{L^{b'}_{x}(\R^{d})}}_{\ell^{2}_{k}(\Z^{3})}.}
The last norm is estimate by the Sobolev embedding $\dot H^{1/2+1/b}_{\bar{b}}(\R^3)\subset L^{b'}(\R^3)$ with $\bar{b}\in(1,\I)$ defined by
\EQ{
 \frac{1}{\bar{b}}:=\frac{1}{b'}+\frac13\sqb{\frac12+\frac1b}=\frac 76-\frac{2}{3b},} 
\EQ{
& \nor{\big\|U^{-1/2-1/b}\psi_{k}|\nabla|\varphi\big\|_{L^{b'}_{x}(\R^{d})}}_{\ell^{2}_{k}(\Z^{3})}
\\&	\lec\nor{\big\|\psi_{k}|\nabla|\varphi\big\|_{L^{\bar{b}}_{x}(\R^{d})}}_{\ell^{2}_{k}(\Z^{3})}
  \lec \||\na|\fy\|_{L^2(\R^3)}.}
Then using \eqref{t power norm}, together with \eqref{emb X} if $a<2$, leads to the desired estimate. 
\end{proof}
\begin{rem}
If we do not exploit the low frequency gain $U^{1/2-1/b}$ for $e^{-itH}$ in the above proof at all, then the upper bound on $b$ becomes $b<3$, which is related to the critical Sobolev embedding $\dot H^1_3(\R^3)\not\subset L^\I(\R^3)$. On the other hand, we need $\e>1/4$ to treat the quadratic nonlinear terms, which can be easily seen by a  scaling argument, but $\e>1/4$ is equivalent to $b>3$ in the above lemma. 
This implies that we can not prove Theorem \ref{thm:main2} simply as for \eqref{eq:NLS} without using the low frequency gain.  
\end{rem}

The next lemma is the main deterministic claim in this section. 
\begin{lem}
\label{lem:contraction2}
Let $1/4<\e<3/8$. Then there exists $\eta_{0}>0$ such that the following holds. 
Let $(q,r)\in(2,\I)\times[12/5,4]$ be an admissible pair on $\R^3$. 
Then for any $\eta\in(0,\eta_0]$, $\z_+\in H^1(\R^3)$ and $T\ge 1$ satisfying 
\EQ{ \label{ze assmall}
 \|\LR{\na}e^{-itH}\z_+\|_{X^{q,r}_\e(I_T)} + \|e^{-itH}\z_+\|_{X^{\I,3}_0(I_T)}\le \eta,}
there exists a unique local solution $\z\in C_t(I_T;B_*)$ to \eqref{eq:ze-g} on $I_T$ such that 
\EQ{ \label{ze Str}
 \|\LR{\na}(\z - e^{-itH}\z_+)\|_{X^{\I,2}_\e(I_T)\cap X^{q,r}_\e(I_T)} \le \eta,}
where $B_*$ is the domain of $g$ defined in Lemma \ref{lem:z-u}. 

\end{lem}
\begin{proof}
The equation \eqref{eq:ze-g} for $\z$ with the final-data $\z_+$ is rewritten as 
\EQ{
 \z = e^{-itH}\z_+ + i\int_t^\I e^{-i(t-s)H}N(g(\z))ds,}
as long as $\z$ stays in $B_*$. Hence it suffices to show that the mapping $\Psi$ defined by 
\EQ{
 \z^0:=e^{-itH}\z_+, \pq \Psi(w) := \z^0 + i\int_t^\I e^{-i(t-s)H}N(g(\z))ds }
is a contraction on the following closed non-empty set in a Banach space: 
\EQ{
 Y_T := \{w\in C(I_T;B_*),\ \|\LR{\na}(w-\z^0)\|_{X^{\I,2}_\e(I_T)\cap X^{q_1,r_1}_\e(I_T)}\le\eta\},}
endowed with the norm $\|\LR{\na}\cdot\|_{X^{\I,2}_\e(I_T)\cap X^{q_1,r_1}_\e(I_T)}$, provided that $\eta_0>0$ is small. 

First we impose $\eta_0\le\eta_*$, where $\eta_{*}$ is as in Lemma \ref{lem:z-u}, then $\z^0\in Y_T$. 
Next to see that $\Psi$ maps $Y_T$ into itself, let $w\in Y_T$ and $u:=g(w)$. 
Then Lemma \ref{lem:z-u} implies that $u\in C(I_T;H^1(\R^3))$, and for all $t\in I_T$, $s\in[0,1]$ and $r\in[2,6]$, 
 we have $\|u(t)\|_{H^s_r(\R^3)} \lec \|w(t)\|_{H^s_r(\R^3)}$. 
In particular, 
\EQ{ \label{est z2u}
 \|\LR{\na}u\|_{X^{q_1,r_1}_\e(I_T)} \lec \|\LR{\na}w\|_{X^{q_1,r_1}_\e(I_T)},}
and, using $H^1(\R^3)\subset L^3(\R^3)$, 
\EQ{ \label{u L3 small}
 \|u\|_{X^{\I,3}_0(I_T)} \lec \|\z^0\|_{X^{\I,3}_0(I_T)} + \|\LR{\na}(w-\z^0)\|_{X^{\I,2}_0(I_T)} \lec \eta.}

On the nonlinear terms, using the Sobolev norm $\|\fy\|_{H^1_r}=\|\LR{\na}\fy\|_{L^r}\sim\|\fy\|_{L^r}+\|\na\fy\|_{L^r}$ 
together with H\"older's inequality as well as the boundedness of $|\na|^{-1}\na$ on $L^r$ (for $1<r<\I$), the quadratic part is estimated at each $t\in I_T$ for $j=0,1$ by 
\EQ{
 &\|u_1^2\|_{H^1_{r/2}} \lec \|u^2\|_{L^{r/2}}+\|u\na u\|_{L^{r/2}}
  \lec\|u\|_{H^1_{r}}^2, 
 \\& \|H^{-1}\na\cdot(u_1\na u_2)\|_{H^1_{r/2}} \lec \|u\na u\|_{L^{r/2}} \lec \|u\|_{H^1_r}^2.}
For the cubic part, we also use the Sobolev inequality:
\EQ{
 \|u^2\|_{L^{r}} \lec \|\na u^2\|_{L^{\bar r}} \lec \|u\|_{L^3}\|\na u\|_{L^{r}},}
where $\bar r\in[6/5,2)$ is defined by $1/\bar r=1/r+1/3$. Then at each $t\in I_T$, we have
\EQ{
 &\||u|^2u_1-iH^{-1}\De(|u|^2u_2)\|_{H^1_{r/2}} \lec \||u|^2u\|_{H^1_{r/2}}
 \\&\lec \|u^3\|_{L^{r/2}} + \|u^2\na u\|_{L^{r/2}}
 \lec \|u^2\|_{L^{r}}(\|u\|_{L^{r}}+\|\na u\|_{L^{r}})
 \\&\lec \|u\|_{L^3} \|\na u\|_{L^{\bar r}} \|u\|_{H^1_{r}} 
 \lec \|u\|_{L^3}\|u\|_{H^1_{r}}^2.}
Then using H\"older in $t$ and \eqref{u L3 small}, we obtain, assuming $\eta_0\le 1$,  
\EQ{
 \|\LR{\na}N(u)\|_{\ti X_T} \lec \|\LR{\na}u\|_{X_T}^2,}
where $\ti X_T:=X^{2q,2r}_{2\e}(I_T)$ and $X_T:=X^{q,r}_\e(I_T)$. 
Since \eqref{dual qr} holds with $d=3$, $p=1$, there is a unique admissible pair $(\hat q,\hat r)$ such that $\hat r'=2r$. Then from \eqref{nonlin decay}, we have $1/\hat q'-1/(2q)=1/4<\e$ and so 
\EQ{
 \|\LR{\na}N(u)\|_{\hat X_T} &\lec T^{1/4-\e}\|\LR{\na}N(u)\|_{\ti X_T} 
 \lec T^{1/4-\e} \|\LR{\na}u\|_{X_T}^2,}
where $\hat X_T:=X^{\hat q',\hat r'}_\e(I_T)$. 
Combining this with \eqref{est z2u} and Lemma \ref{lem:t-Strichartz 4H}, we obtain 
\EQ{
& \|\LR{\na}(\z^0-\Psi(w))\|_{X^{\I,2}_\e(I_T)\cap X_T}
\\& \lec T^{1/4-\e}(\|\LR{\na}\z^0\|_{X_T}+\|\LR{\na}w\|_{X_T})^2 \lec T^{1/4-\e}\eta^2,}
which is much smaller than $\eta$ if $\eta_0$ is small. Also using $H^1(\R^3)\subset L^3(\R^3)$,  
\EQ{
 \|\Psi(w)\|_{X^{\I,3}_0(I_T)} \le \eta + CT^{1/4-\e}\eta^2 \le \eta_*,}
if $\eta_0$ is small enough. Hence $\Psi$ maps $Y_T$ into itself. 

To show that it is a contraction, let $w^0,w^1\in Y_T$, $u^j:=g(w^j)$, and $u':=u^0-u^1$. 
Since $N(u)$ consists of Fourier multipliers, product and sum, $N(u^0)-N(u^1)$ is expanded into similar quadratic and cubic terms in $u^0,u^1$ and $u'$, which are at least linear in $u'$ (and $\bar u'$). 
Hence by the same argument, we obtain 
\EQ{
 \|\LR{\na}(N(u^0)-N(u^1))\|_{\ti X_T}
 \lec \eta\|\LR{\na}u'\|_{X^{\I,2}_0(I_T)\cap X_T},}
where the $X^{\I,2}_0(I_T)$ norm is used to bound $\|u'\|_{L^3_x}$ by Sobolev. 
By Lemma \ref{lem:z-u}, we have 
\EQ{
 \|\LR{\na}u'\|_{X^{\I,2}_0(I_T)\cap X_T} \lec \|\LR{\na}(w^0-w^1)\|_{X^{\I,2}_0(I_T)\cap X_T}.}
Then using the same linear estimates as above, we obtain 
\EQ{
 & \|\LR{\na}(\Psi(w^0)-\Psi(w^1))\|_{X^{\I,2}_\e(I_T)\cap X_T}
 \\&\lec T^{1/4-\e}\|\LR{\na}(N(u^0)-N(u^1))\|_{\ti X_T}
 \\&\lec T^{1/4-\e}\eta\|\LR{\na}(w^0-w^1)\|_{X^{\I,2}_\e(I_T)\cap X_T}.}
Taking $\eta_0>0$ smaller if necessary, we deduce that $\Psi$ is a contraction on $Y_T$, so there is a unique fixed point $w\in Y_T$, which is the unique solution.
\end{proof}

\begin{rem}
Since we have $v=\z-\LR{\na}^{-1}|\na|^{-1}|u|^{2}$ for $v=U^{-1}u_{1}+iu_{2}$,
it follows that
\EQ{
\|v-e^{-itH}\z_{+}^{\omega,1}\|_{H^{1}(\R^{3})}
&\lec \|\z-e^{-itH}\z_{+}^{\omega,1}\|_{H^{1}(\R^{3})}+\||u|^{2}\|_{\dot{H}^{-1}(\R^{3})}.}
Combining that
\EQ{
\||\na|^{-1}|u|^{2}\|_{X^{\infty,2}_{0}(I_{T})}
&\lec \||u|^{2}\|_{X^{\infty,6/5}_{0}(I_{T})}
\lec T^{-\epsilon}\|u\|_{X^{\infty,2}_{\epsilon}(I_{T})}\|u\|_{X^{\infty,3}_{0}(I_{T})}
\\&\lec T^{-\epsilon}\|\LR{\na}w\|_{X^{\infty,2}_{\epsilon}(I_{T})}\|u\|_{X^{\infty,3}_{0}(I_{T})}
\lec T^{-\epsilon}\eta^{2},}
we obtain that $\|v-e^{-itH}\z_{+}^{\omega,1}\|_{H^{1}(\R^{3})}=o(t^{-\epsilon})$ as $t\to\infty$.
\end{rem}

From Lemmas \ref{prop:random-L2-range}, \ref{lem:a.s.linear02} and \ref{lem:contraction2},
we can prove Theorem \ref{thm:main2}.

\begin{proof}[Proof of Theorem \ref{thm:main2}] 
Take any $\e\in(1/4,3/8)$ and an admissible pair $(q,r)\in(2,\I)\times[12/5,4]$ as in Lemma \ref{lem:contraction2}, and let $\eta_{0}>0$ be given by the lemma. Then 
\EQ{
 \e_0:=-\e+3/2-3/r-1/q=-\e+1/q_0>0,}
so by Lemma \ref{lem:a.s.linear02},
for all $\z_{+}\in \dot{H}^{1}(\R^{3})$ and for almost every $\om\in\Omega$, 
we have $e^{-itH}\z_{+}^{\omega,1}\in X^{q,r}_\e(I_1)$. 
Moreover, since $\z_{+}^{\om,1}\in H^{1}(\R^{d})$ for almost every $\omega\in\Omega$, we have
\begin{align}
&
\lim_{T\to\infty}\nor{e^{-itH}\z_{+}^{\om,1}}_{X^{\infty,3}_{0}(I_{T})}=0
\end{align}
by Lemma \ref{lem:H-U-P} together with the Sobolev embedding $H^1(\R^3)\subset L^3(\R^3)$. 
Therefore, for any $\eta\in(0,\eta_0)$ and for almost every $\omega\in\Omega$, 
there exists $T>0$ such that 
\EQ{
  \|\LR{\na}e^{-itH}\z_+^{\om,1}\|_{X^{q,r}_\e(I_T)} + \|e^{-itH}\z_+^{\om,1}\|_{X^{\I,3}_0(I_T)}\le \eta,}
thus by Lemma \ref{lem:contraction2}, there exists a unique solution
$\z\in C(I_T;B_*)$ to \eqref{eq:ze-g} on $I_{T}$ satisfying
\EQ{
  \|\LR{\na}(\z - e^{-itH}\z_+^{\om,1})\|_{X^{\I,2}_\e(I_T)\cap X^{q,r}_\e(I_T)} \le \eta.}
Applying the weighted Strichartz estimate once again to $\Psi(\z)$, we obtain the same estimate on all the other admissible norms. 
In particular, it is scattering in $H^{1}(\R^{3})$ with the final-state $\z_+^{\om,1}$. 
Since $u:=g(z)\in C(I_T;H^1(\R^3))$ is a solution of \eqref{eq:GP-u} on $I_T$, the global well-posedness in \cite{B-S} implies that $u$ is extended uniquely to a global solution of \eqref{eq:GP-u} in $C(\R;H^1(\R^3))$. 
\end{proof}
We can prove that $v=U^{-1}u_1+iu_2$ remains in the Strichartz space after the extension of $u$ to all $t\in\R$,  namely $\LR{\na}v\in L^q_{t,loc}L^r_x(\R\times\R^d)$, which is a stronger condition than that implied by the global well-posedness for $u(t)\in H^1(\R^3)$. The detail is however omitted, since the main question in this paper is the uniqueness, for which the weaker condition is better. 

%%%%%%%%%%%%%%%%%%%%%%%%%%%%%%%%%%%%%%%%%%%%%%%%%%%%%%%%%%%%
%%%%%%%%%%%%%%%%%%%%%%%%%%%%%%%%%%%%%%%%%%%%%%%%%%%%%%%%%%%%
%%%%%%%%%%%%%%%%%%%%%%%%%%%%%%%%%%%%%%%%%%%%%%%%%%%%%%%%%%%%

\end{document}